\documentclass[11 pt]{amsart}

\usepackage{amsthm,amssymb,amstext,amscd,amsfonts,amsbsy,amsxtra,latexsym,cite}
\usepackage[latin1]{inputenc}
\usepackage{graphicx, graphics, wrapfig}

\newcommand{\field}[1]{\mathbb{#1}}

\newcommand{\Z}{\field{Z}}
\newcommand{\R}{\field{R}}

\newcommand{\Gbar}[1]{\overline{G}_{#1}}

\newcommand{\pbar}[1]{\overline{p}_{#1}}

\newcommand{\op}{\overline{p}}

\def\({\left(}
\def\){\right)}
\newcommand{\df}{\emph}
\DeclareMathOperator{\Li}{Li}

\renewcommand{\o}{\overline}

\theoremstyle{plain}
\newtheorem{theorem}{Theorem}
\newtheorem*{theorem*}{Theorem}

\newtheorem{proposition}[theorem]{\textbf{Proposition}}

\newtheorem*{conjecture*}{Conjecture}

\theoremstyle{definition}

\theoremstyle{remark}
\newtheorem*{example}{Example}
\newtheorem*{remark}{Remark}
\newtheorem*{remarks}{Remarks}

\numberwithin{theorem}{section} \numberwithin{equation}{section}

\begin{document}

\title{$k$-run overpartitions and mock theta functions}
\author{Kathrin Bringmann}
\address{Mathematical Institute\\University of
Cologne\\ Weyertal 86-90 \\ 50931 Cologne \\Germany}
\email{kbringma@math.uni-koeln.de}
\author{Alexander E. Holroyd}
\address{Microsoft Research \\ 1 Microsoft Way \\ Redmond, WA 98052, USA}
\email{holroyd at microsoft.com}
\author{Karl Mahlburg}
\address{Department of Mathematics \\
Princeton University \\
NJ 08544\\ U.S.A.}
\email{mahlburg@math.princeton.edu}
\author{Masha Vlasenko}
\email{masha.vlasenko@gmail.com}
\date{\today}
\thanks{The research of the first author was supported by the Alfried Krupp Prize for
Young University Teachers of the Krupp Foundation.
The third author was supported by an NSF Postdoctoral Fellowship
administered by the Mathematical Sciences Research Institute through
its core grant DMS-0441170.  The first and third authors were also partially supported by the Alexander von Humboldt Foundation}

\begin{abstract}
In this paper we introduce $k$-run overpartitions as natural analogs to partitions without $k$-sequences, which were first defined and studied by Holroyd, Liggett, and Romik.  Following their work as well as that of Andrews,
we prove a number of results for $k$-run overpartitions, beginning with a double summation $q$-hypergeometric series representation for the generating functions.  In the
special case of $1$-run overpartitions we further relate the generating function to one of Ramanujan's mock theta functions.  Finally, we describe the relationship between $k$-run overpartitions and certain sequences of random events, and use probabilistic estimates in order to determine the asymptotic growth behavior of the number of $k$-run overpartitions of size $n$.
\end{abstract}

\maketitle

\section{Introduction and statement of results}

A {\it partition} of a positive integer $n$ is a non-increasing sequence of
positive integers whose sum is $n$; the number of distinct partitions of $n$
is traditionally denoted by $p(n)$.  A {\it sequence} (or {\it run}) in a
partition is any subsequence of consecutive integers that appear as parts.
Integer partitions without sequences were first studied by MacMahon in
\cite{Mac16}.  He described their relationship to partitions with repeated
parts, and also determined their generating series. These partitions were
studied more recently by Holroyd, Liggett, and Romik in \cite{HLR}, where the
authors introduced the general family of {\it partitions without
$k$-sequences} for any $k \geq 2$, in which no $k$ consecutive integers may
all appear as parts. The number of partitions of $n$ without $k$-sequences of
$n$ is denoted by $p_k(n)$, and the generating function is defined as
$$
G_k(q) := \sum_{n \geq 0} p_k(n) q^n.
$$
These partitions were also studied by Andrews \cite{And05}, who found a
(double) $q$-hypergeometric series expansion for the generating function.
Before giving this series, we record the definition for the $q$-Pochhammer
symbol, which is given by $(a;q)_n := \prod_{j = 0}^{n-1} (1- aq^j).$

Andrews showed that
\begin{equation}
\label{E:Gkdouble}
G_k(q) = \sum_{r, s \geq 0} \frac{(-1)^r\,
q^{\textstyle \frac{(k+1)k(r+s)^2}{2} + \frac{(k+1)(s+1)s}{2}}}
{\left(q^k; q^k\right)_r \left(q^{k+1}; q^{k+1}\right)_s}.
\end{equation}

Both of these earlier papers also addressed the asymptotic
behavior of $p_k(n)$ as $n \to \infty$.  Holroyd, Liggett, and
Romik \cite{HLR} showed the asymptotic expression
\begin{equation}
\label{E:logpk}
\log{p_k(n)} \sim \pi \sqrt{\frac{2}{3}\left(1 - \frac{2}{k(k+1)}\right)n}
\quad\text{as }n\to\infty.
\end{equation}
Andrews \cite{And05} substantially improved this result in the
case $k=2$, proving the asymptotic expression
\begin{equation}
\label{E:p2}
p_2(n) \sim \frac{1}{4 \sqrt{3} n^{\frac34}} e^{\frac{2 \pi}{3}\sqrt{n}}
\quad\text{as }n\to\infty.
\end{equation}
His proof relies on a beautiful and surprising
relationship between $G_2(q)$ and one of Ramanujan's famous
mock theta functions; we will further discuss this connection below.
\begin{remark}
The above asymptotic expression is not stated as such in Andrews' paper. He
instead focused on the asymptotics of the generating series $G_2(q)$ as $q\to
1$.  However, his results directly imply \eqref{E:p2} upon applying a
Tauberian theorem (cf. Section \ref{S:OverpartnAsymp} of this paper).
Recently \eqref{E:p2} was greatly refined by the first and third authors of
the present paper in \cite{BM}.  Specifically, the latter paper introduced a
generalization of the Circle Method in order to prove a series expansion for
$p_2(n)$ involving Kloosterman sums, Bessel functions, and principal value
integrals of modified Bessel functions.
\end{remark}

In this paper we study a related family of overpartitions.
As introduced by Corteel and Lovejoy in \cite{CL}, an
\df{overpartition} is a partition in which the last occurrence
of each part may be overlined.  The number of overpartitions of
size $n$ is denoted by $\op(n)$.  An overpartition is said to
have a \df{gap} at $m$ if there are no parts of size $m$.

We define \df{lower $k$-run overpartitions} to be those
overpartitions in which any overlined part must occur within a
run of exactly $k$ consecutive overlined parts that terminates
below with a gap.  More precisely, this means that if some part
$\overline{m}$ is overlined, then there is an integer $j$ with
$m\in[j+1,j+k]$ such that each of the $k$ overlined parts
$\o{j+1}, \o{j+2}, \ldots, \o{j+k}$ appear (perhaps together
with non-overlined versions), while no part $j$ (overlined or
otherwise) appears, and no overlined part $\o{j+k+1}$ appears.
There is a simple bijection between lower $k$-run
overpartitions and \df{upper $k$-run overpartitions}, which are
defined analogously but with the conditions on parts $j$ and
$j+k+1$ reversed (see Section \ref{S:Combinatorial}). Denote
the number of lower $k$-run overpartitions of size $n$ by
$\pbar{k}(n)$.

\begin{example}
The lower $2$-run overpartitions of size $7$ are
\begin{gather*}
\o{4}+\o{3},\quad 4+\o{2}+\o{1},\quad \o{3}+2+\o{2},\quad
3+\o{2}+1+\o{1}, \\ 2+2+\o{2}+\o{1},\quad
2+\o{2}+1+1+\o{1},\quad \o{2}+1+1+1+1+\o{1},
\end{gather*}
together with the $15$ partitions of $7$, so $\o{p}_2(7)=7+15 =22$.
\end{example}

The generating function for lower $k$-run
overpartitions is denoted by
$$
\Gbar{k}(q) := \sum_{n \geq 0} \o{p}_k(n) q^n.
$$
Our first result is analogous to Andrews' double-series
generating function \eqref{E:Gkdouble} for partitions without
$k$-sequences.
\begin{theorem}
\label{T:Gbarq}
For $|q| < 1$,
\begin{equation*}
\Gbar{k}(q) = \frac{1}{(q;q)_\infty}
\sum_{r, s \geq 0} \frac{(-1)^s \,
q^{\textstyle \frac{(k+1)k(r+s)^2}{2} + \frac{(k+1)s(s+1)}{2}}}
{(q^{k};q^{k})_r (q^{k+1}; q^{k+1})_s}.
\end{equation*}
\end{theorem}

Our next result is an asymptotic expression for lower $k$-run overpartitions
that is much stronger than the logarithmic expression in \eqref{E:logpk}.
\begin{theorem}
\label{T:pk}
As $n \rightarrow \infty$,
$$
\pbar{k}(n)\sim \frac{1}{2\sqrt{6}n} \sqrt{1 + \frac{1}{2k(k+1)}} \exp\left(\pi \sqrt{\frac{2}{3} \left(1 + \frac{1}{2k(k+1)}\right)n} \right).
$$
\end{theorem}
\begin{remark}
Interestingly, our techniques do not apply to the case of partitions without
sequences, despite the similarity of \eqref{E:Gkdouble} and Theorem
\ref{T:Gbarq}.
\end{remark}

We next focus particularly on certain special cases that are related to Ramanujan's
mock theta functions.   Andrews \cite{And05} showed
that the generating function for partitions without $2$-sequences may be
written as
\begin{equation}
G_2(q) = \frac{\left(-q^3; q^3\right)_\infty}{\left(q^2; q^2\right)_\infty} \cdot \chi(q),
\end{equation}
where
\begin{equation*}
\chi(q) := 1 + \sum_{n \geq 1} \frac{q^{n^2}}{\prod_{j = 1}^n \left(1 - q^j + q^{2j}\right)},
\end{equation*}
which is one of Ramanujan's third-order mock theta functions. Ramanujan
originally introduced the mock theta functions by listing a small
collection of examples in his last letter to Hardy \cite{Wat}. He justified
his own interest by observing their striking asymptotic properties and
near-symmetries under modular transformations.  Andrews \cite{And05} used
some of these properties in order to determine the asymptotic behavior of
$G_2(q)$ as $q \to 1$, which then implies \eqref{E:p2}.  The general theory
of mock theta functions has also recently seen great advancements, as
Zwegers' doctoral thesis \cite{Zw02} has led to a proper understanding of the
relationship to automorphic forms \cite{BO, BO2, Zag06}.

The case $k=1$ of lower $k$-run overpartitions is similarly related to
another of Ramanujan's third-order mock theta functions from \cite{Wat}.  In
particular, the mock theta function
\begin{equation}
\label{E:phi}
\phi(q) := 1 + \sum_{n \geq 1} \frac{q^{n^2}}{\left(-q^2; q^2\right)_n},
\end{equation}
appears in the following expression for the generating function of lower
$1$-run overpartitions.
\begin{theorem}
\label{P:Gbar1}
For $|q| < 1$,
\begin{equation*}
\Gbar{1}(q) = (q; q)_\infty \cdot \phi(q).
\end{equation*}
\end{theorem}
\begin{remark}
Note that the $k=1$ case of Theorem \ref{T:pk} states that
$\overline{p}_1(n) \sim \frac{\sqrt{5}}{4\sqrt{6}\,n} e^{\pi\sqrt{\frac{5n}{6}}}.$
\end{remark}

The remainder of the paper is structured as follows.  In Section
\ref{S:Overpartitions}, we  consider basic combinatorial properties of
$k$-run overpartitions and derive their generating series.  In Section
\ref{S:Constant}, we apply the Constant Term Method to determine the
asymptotic behavior of the generating series.  Finally, in Section
\ref{S:OverpartnAsymp}, we prove the asymptotic expression for $k$-run overpartitions found in Theorem \ref{T:pk}.

\section{Overpartition combinatorics and generating series}
\label{S:Overpartitions}

\subsection{Combinatorial results for $k$-run overpartitions}
\label{S:Combinatorial}

In this section we denote the number of lower $k$-run
overpartitions of $n$ by $\pbar{k}^{(-)}(n)$.  We also recall the definition
\df{upper $k$-run overpartitions}  as overpartitions in which
the overlined parts must occur in consecutive runs of the form
$\o{j+1}, \ldots, \overline{j+k}$, with no overlined parts
$\o{j}$, and no parts of any kind of size $j+k+1$. The number
of upper $k$-run overpartitions of size $n$ is denoted by
$\pbar{k}^{(+)}(n)$.

Our first observation on $k$-run overpartitions is that the
lower and upper definitions are in bijective correspondence.
\begin{proposition}
\label{P:op-+} For all $n\geq 0$ and $k \geq 1$,
$$\pbar{k}^{(-)}(n) = \pbar{k}^{(+)}(n).$$
\end{proposition}

\begin{proof}
We construct a simple bijection between the two sets of
overpartitions.  In a lower $k$-run overpartition, any run of
$k$ overlined parts occurs at the lower end of some block of
consecutive parts, surrounded on both sides by gaps.  That is,
for some $j$ and some $\ell\geq j+k$, all the parts
$$\o{j+1},\ldots,\o{j+k},j+k+1,\ldots,\ell$$
appear (perhaps together with further non-overlined copies of
these same values), with gaps at $j$ and $\ell+1$.  We can form
a new overpartition by shifting the $k$ overlines to the upper
end of the block, i.e.,\ replacing the above parts with
$$j+1,\ldots,\ell-k,\o{\ell-k+1},\ldots,\o\ell.$$
Applying this transformation to every run of $k$ overlined parts
results in an upper $k$-run overpartition, and this map is
clearly bijective.
\end{proof}

We henceforth write simply $\pbar{k}$ for
$\pbar{k}^{(-)}$, and abbreviate the term ``lower $k$-run overpartition'' to
\df{$k$-run overpartition}.  It will be convenient to work with
both the upper and lower definitions in different contexts. In
view of Proposition~\ref{P:op-+} they are interchangeable for
purposes of enumeration.

We have the following monotonicity properties for $k$-run overpartitions.
\begin{proposition}
\label{P:mono}
For any $n \geq 0$ and $k \geq 1$,
\begin{enumerate}
\item
\label{P:mono:n}
$\pbar{k}(n)\leq \pbar{k}(n+1)$
\item
\label{P:mono:k}
$\pbar{k}(n) \geq \pbar{k+1}(n)$.
\end{enumerate}
\end{proposition}

\begin{proof}
\ref{P:mono:n}  Given an upper $k$-run overpartition of $n$, adding a
non-overlined part $1$ clearly gives an upper $k$-run
overpartition of $n+1$, and this map is injective.

\ref{P:mono:k}  We construct an injection from lower $(k+1)$-run
overpartitions to lower $k$-run overpartitions.  If $\lambda$
is a $(k+1)$-run overpartition and $\overline{j+1}, \dots,
\overline{j+k+1}$ is a $(k+1)$-run, then there are no parts of
size $j$, and there can only be non-overlined parts of size
$j+k+2$.  Removing the overline on $\o{j+k+1}$ results in a
$k$-run satisfying the lower $k$-run condition, and applying
this to every run gives the required injection.
\end{proof}
\begin{remarks}\
\begin{enumerate}
\item We note that the partitions without sequences studied in \cite{HLR}
    satisfy the opposite inequality to that in Proposition \ref{P:mono}
    \ref{P:mono:k}:
\[
p_k(n)\leq p_{k+1}(n).
\]
\item For any fixed $n$, we have $\pbar{k}(n) = p_k(n) =
    p(n)$ for sufficiently large $k$.
\end{enumerate}
\end{remarks}

\subsection{A double summation series}
In this section we consider the two-variable generating
functions
\begin{equation}
\Gbar{k}(x) = \Gbar{k}(x; q)
:= \sum_{\ell, n \geq 0} \pbar{k}(\ell,n) x^\ell q^n,
\end{equation}
where $\pbar{k}(\ell,n)$ is defined to be the number of $k$-run
overpartitions of size $n$ and exactly $\ell$ parts.  The main
result in this section is a double series expansion for this
generating function, which is proven from a $q$-difference
equation for $\Gbar{k}(x; q)$.

\begin{theorem}
\label{T:Gbarxq}
For $|q| < 1$,
\begin{equation*}
\Gbar{k}(x;q) = \frac{1}{(xq;q)_\infty} \sum_{r, s \geq 0}
\frac{(-1)^s x^{kr + (k+1)s}
q^{\textstyle\frac{k(k+1)(r+s)^2}{2} + \frac{(k+1)s(s+1)}{2}}}
{(q^{k};q^{k})_r (q^{k+1}; q^{k+1})_s}.
\end{equation*}
\end{theorem}

\begin{proof}
The generating function $\overline{G}_k$ satisfies the $q$-difference
equation
\begin{equation}
\label{E:Gbardiff}
\Gbar{k}(x) = \frac{1}{1-xq} \: \Gbar{k}(xq) +
\frac{x^{k}q^{\frac{k(k+1)}{2}}}{(xq;q)_{k}} \: \Gbar{k}\bigl(xq^{k+1}\bigr),
\end{equation}
which follows from the definition of (upper) $k$-run overpartitions
by separating according to their smallest parts.  The
first term corresponds to the overpartitions without an
overlined $\overline{1}$, in which case there is no restriction
on subsequent parts; the second term corresponds to
overpartitions that do have the part $\overline{1}$, and thus
also have each of the parts $\overline{2}, \dots,
\overline{k}$ (and possibly non-overlined parts of these
sizes as well), followed by no parts of size $k+1$.

This recurrence is easier to solve after ``re-normalizing'' the
equation, which is achieved by setting
\begin{equation}
\label{E:Lk}
\overline{L}_k(x) = \overline{L}_k(x;q) := (xq;q)_\infty \Gbar{k}(x;q).
\end{equation}
Equation \eqref{E:Gbardiff} then implies that this normalized
function satisfies the $q$-difference equation
\begin{equation}
\label{E:Lbardiff}
\overline{L}_k(x) - \overline{L}_k(xq)
= x^{k}q^{\frac{k(k+1)}{2}}\bigl(1-xq^{k+1}\bigr) \overline{L}_k\bigl(xq^{k+1}\bigr).
\end{equation}
We next write this as a series in $x$, where the coefficients
are polynomials in $q$,  $\overline{L}_k(x;q) =:
\sum_{m \geq 0} \overline{\lambda}_m(q) x^m.$ Equation
\eqref{E:Lbardiff} is equivalent to the statement that for all
$m$,
\begin{equation}
\label{E:calF}
\left(1-q^m\right) \: \overline{\lambda}_m = q^{m(k+1) - \frac{k(k+1)}{2}}\left(\overline{\lambda}_{m-k}-\overline{\lambda}_{m-k-1}\right).
\end{equation}

By \eqref{E:Lk}, the theorem statement is equivalent to verifying that $\overline{L}_k(x;q)$ is equal to the double series
\begin{equation*}
\sum_{r, s \geq 0} \frac{(-1)^s x^{kr + (k+1)s}q^{\textstyle\frac{k(k+1)(r+s)^2}{2}
+ \frac{(k+1)s(s+1)}{2}}}{(q^{k};q^{k})_r (q^{k+1}; q^{k+1})_s}.
\end{equation*}
It is sufficient to show that this double series satisfies the $q$-difference
equation \eqref{E:Lbardiff}, while also checking that the $k+1$ initial
$x$-powers are compatible. Denote the $x^m$ coefficient in the double series
by $\widetilde{\lambda}_m$; it is clear by definition that
$\widetilde{\lambda}_0 = 1$ and that $\widetilde{\lambda}_1 = \dots =
\widetilde{\lambda}_{k-1} = 0$, and finally
$$
\widetilde{\lambda}_k = \frac{x^kq^{\frac{k(k+1)}{2}}}{1 - q^k}.
$$
Note that $k$-run overpartitions with at most $k-1$ parts can contain
only non-overlined parts, and thus
\begin{equation*}
\Gbar{k}(x;q) = \frac{1}{(xq;q)_\infty} + \frac{x^k q^{\frac{k(k+1)}{2}}}{1-q^k} + O\bigl(x^{k+1}\bigr).
\end{equation*}
The first term gives all overpartitions with no overlined parts, and the
second term gives all overpartitions that consist of a single $k$-run of
overlined parts. Thus
\begin{equation*}
\overline{L}_k(x;q) = 1 + \frac{x^k q^{\frac{k(k+1)}{2}}}{1-q^k} + O\bigr(x^{k+1}\bigl),
\end{equation*}
and the initial terms $\lambda_i$ $(0 \leq i \leq k)$ agree with
those given for $\widetilde{\lambda}_i$ above.

Now we verify that the  $\widetilde{\lambda}_i$ satisfy
\eqref{E:calF}.  Temporarily writing
$$
A_{r,s}=A_{r,s}(q):=\frac{(-1)^s q^{\textstyle\frac{k(k+1)(r+s)^2}{2}
+ \frac{(k+1)s(s+1)}{2}}}{\left(q^k;q^{k}\right)_r \left(q^{k+1}; q^{k+1}\right)_s},
$$
we then have $\widetilde{\lambda}_m = \sum_{kr + (k+1)s = m} A_{r,s},$
and thus
\begin{align*}
\left(1-q^m\right) \widetilde{\lambda}_m &= \sum_{\substack{r, s \geq 0: \\ kr + (k+1)s = m}}
A_{r,s}\cdot \left[1 - q^{(k+1)s} + q^{(k+1)s}\left(1-q^{kr}\right)\right]
\\
&= \sum_{\substack{r, s \geq 0: \\ kr + (k+1)s = m - k - 1}} A_{r,s} \cdot
q^{{\textstyle\frac{k(k+1)(2r + 2s + 1)}{2}}+ (k+1)(s+1)} \\
&\qquad+ \sum_{\substack{r, s \geq 0: \\ kr + (k+1)s = m - k}}
A_{r,s} \cdot q^{{\textstyle \frac{k(k+1)(2r + 2s + 1)}{2}}
+(k+1)s}
\\
&= q^{m(k+1) - \frac{k(k+1)}{2}}
 \left(\widetilde{\lambda}_{m-k} - \widetilde{\lambda}_{m-k-1}\right).
\end{align*}
The two sums in the second line were shifted by $s \mapsto s+1$ and $r \mapsto r+1$, respectively.
This completes the proof of the double series formula.
\end{proof}

\subsection{Relation to Ramanujan's mock theta functions}
\label{S:k=2} In this section we prove Theorem~\ref{P:Gbar1}. Recall the
earlier definition of Ramanujan's third-order mock theta function
\begin{equation*}
\phi(q) := \sum_{n \geq 0} \frac{q^{n^2}}{\left(-q^2; q^2\right)_n}.
\end{equation*}
This was proven by Fine (see equation (26.32) of \cite{Fine}) to have the
equivalent form
\begin{equation}
\label{E:phiFine}
\phi(q) = \frac{1}{(q;q)_\infty}
\biggl(1 + 2 \sum_{n \geq 1} \frac{q^n}{1-q^n}\prod_{j = 1}^{n-1} \frac{1 + q^{2j}}{1-q^j}\biggr).
\end{equation}

\begin{proof}[Proof of Theorem~\ref{P:Gbar1}]
We first describe a family of overpartitions that are in bijective
correspondence with lower $1$-run overpartitions. The bijection will be the
{\it conjugation} involution. An overpartition may be represented by a {\it
Ferrer's diagram} that lists each part as a left-justified row of dots, with
the parts listed in decreasing order from top to bottom. Furthermore, since
an overlined part is the last occurrence of any part size, we distinguish it
by marking the corresponding bottom-right ``corner''.  The conjugation map
then simply interchanges the rows and columns of such a diagram.

Under overpartition conjugation, lower $1$-run overpartitions map to
overpartitions in which overlined parts may not occur alone, except for
possibly in the largest part.  In other words, such overpartitions have the
property that if $\overline{m} \in \lambda$ and $m$ is not the largest part
size, then $m \in \lambda$ as well.  See Figure \ref{F:Conj} for an example
of the conjugate of a $1$-run overpartition.

\begin{figure}[here]
\begin{center}
\includegraphics[width = 300 pt]{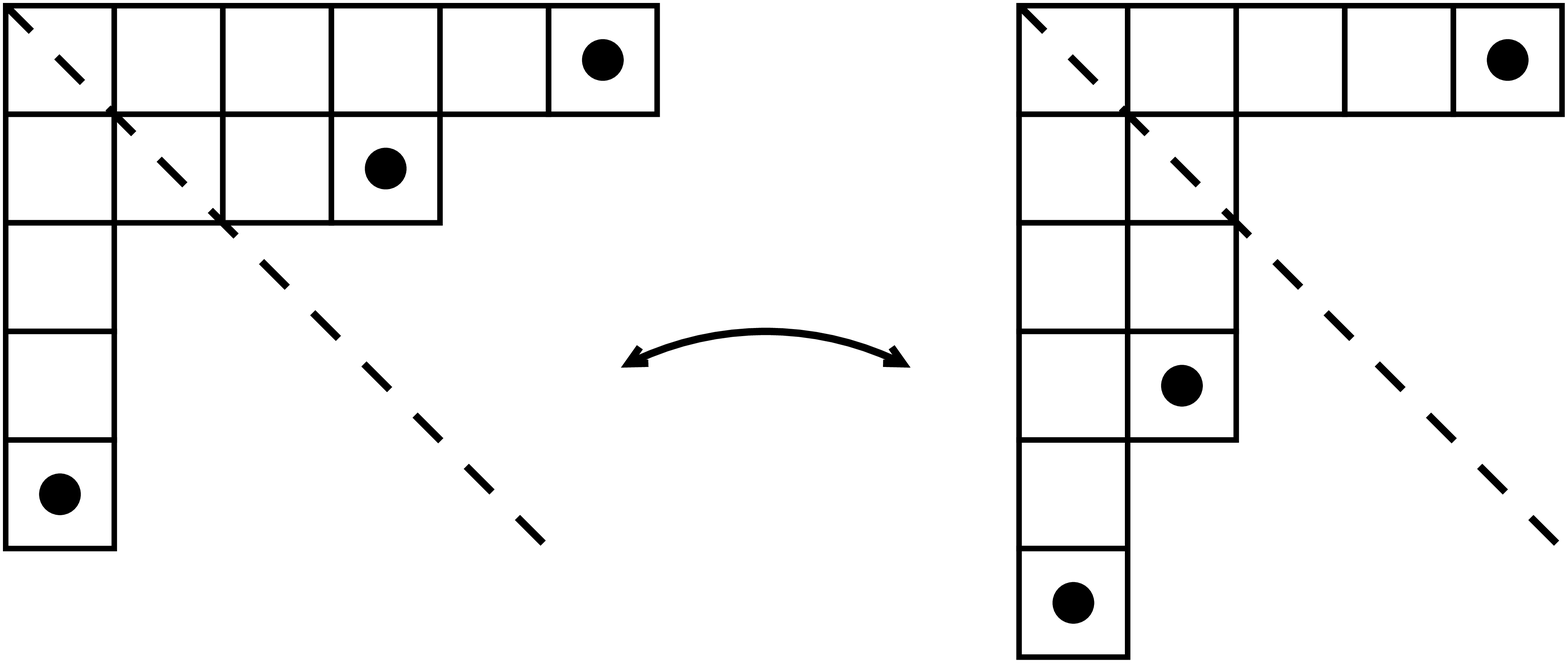}
\end{center}
\caption{\label{F:Conj}
The overpartition on the left, $\overline{6} + \overline{4} + 1 + 1 + \overline{1}$, is a lower $1$-run overpartition.  On the right is its conjugate, $\overline{5} + 2 + 2 + \overline{2} + 1 + \overline{1}$, which has the property that (with the exception of the largest part), no overlined part size occurs alone.}
\end{figure}

By distinguishing the largest part, the generating function for such partitions is clearly
\begin{equation*}
\Gbar{1}(q) = 1 + 2 \sum_{n \geq 1} \left(q^n + q^{2n} + \cdots \right) \prod_{j = 1}^{n-1}
 \left(1 + q^j + 2q^{2j} + 2q^{3j} + \cdots \right),
\end{equation*}
which is easily seen to be equal to the second factor in \eqref{E:phiFine}.
\end{proof}

\section{Generating series and the Constant Term Method}
\label{S:Constant}

Throughout this section we view $q$-series as Fourier expansions.  In fact, we will mainly focus on real values of $q \in (0,1)$, so we write $q := e^{- \varepsilon}$ with $\varepsilon \in \R^+.$  Consider the double hypergeometric $q$-series from Theorem \ref{T:Gbarq}, which we denote as
\begin{align}
\label{E:Hkq}
H_k(q):=
\sum_{r, s\geq 0}\frac{(-1)^s q^{\textstyle \frac{k(k+1)(r+s)^2}{2}+\frac{(k+1)s(s+1)}{2}}}
{\big(q^k; q^k\big)_r\big(q^{k+1}; q^{k+1}\big)_s}
 = (q; q)_\infty \overline{G}_k(q). \notag
\end{align}
\begin{proposition}
\label{P:Hkq}
As $\varepsilon \to 0^+$,
\begin{align*}
H_k(q) =
\sqrt{2}\, e^{\textstyle\frac{\pi^2}{12k(k+1)\varepsilon}}
\left(1 + O\left(\varepsilon^{\frac12}\right)\right).
\end{align*}
\end{proposition}

\begin{proof}
The proof follows the {\it constant-term} method, which begins with the observation that $H_k$ can be written as a (Laurent) coefficient of a two-variable series, namely
\begin{equation}\label{E:Hkqcoeff}
H_k(q) =\mbox{coeff }{\left[x^0\right]}\left(\
\sum_{n\in\Z}q^{\frac{k(k+1)n^2}{2}}x^{-n}
\sum_{r\geq 0}\frac{x^r}{\big(q^k; q^k\big)_r}
\sum_{s\geq 0}\frac{(-1)^s x^s q^{\frac{(k+1)s(s+1)}{2}}}{\big(q^{k+1}; q^{k+1}\big)_s}
\right).
\end{equation}
As indicated, the series on the right side of \eqref{E:Hkqcoeff} factorizes into three sums (in $n$, $r$, and $s$), and we now express each of the three factors in terms of known number-theoretic
functions. Note that all of the summations are convergent for $|x|,|q|<1$.
The first factor is $\theta\left(q^{\frac{k(k+1)}{2}};x\right)$, where
\begin{equation*}
\theta(q;x) := \sum_{n \in \Z} q^{n^2}x^n
\end{equation*}
is the {\it Jacobi theta function}.

The second and third summations can both be expressed as infinite products
via the following instances of the $q$-binomial theorem \cite{And}:
\begin{align*}
\frac{1}{(x; q)_\infty} &= \sum_{m\geq 0}\frac{x^m}{(q; q)_m},\\
(x; q)_\infty &= \sum_{m\geq 0}\frac{(-1)^m q^{\frac{m(m-1)}{2}}x^m}{(q; q)_m}.
\end{align*}
We further rewrite these terms using the {\it quantum dilogarithm}; this function is defined by
\begin{equation*}
\Li_2(x; q) := -\log (x; q)_\infty = \sum_{m \ge 1} \frac{x^m}{m(1-q^m)}.
\end{equation*}

The expression in \eqref{E:Hkqcoeff} can now be written as
\[
H_k(q) = \text{coeff }\left[x^0\right]\Bigg( \theta\left(q^{\frac{k(k+1)}2};x\right)\,
\exp\bigg(\Li_2\left(x;q^k\right)-\Li_2\left(x q^{k+1}; q^{k+1}\right)\bigg)\Bigg).
\]
Applying Cauchy's theorem, we recover the $x^0$-coefficient as
\begin{align*}
H_k(q)= \int_{[0, 1]+ic} \theta\left(q^{\frac{k(k+1)}2};e^{2 \pi i u}\right)\,
 \exp\bigg(\Li_2\left(e^{2 \pi i u};q^k\right)-
 \Li_2\left(e^{2 \pi i u} q^{k+1}; q^{k+1}\right)\bigg) du,
\end{align*}
where $c > 0$ is a real constant that will be specified later.  The theta function may be transformed
using the standard Poisson summation formula \cite{Ko}, resulting in
\begin{equation*}
\theta\left(e^{ -\frac{k(k+1)\varepsilon}{2}};e^{2 \pi i u} \right)
 = \sum_{n\in\Z}e^{ {\textstyle\frac{-k(k+1)n^2\varepsilon }{2} }- 2\pi i n u}
 =\sqrt{\frac{2\pi}{\varepsilon k(k+1)}}\sum_{n\in\Z}e^{-\textstyle\frac{2\pi^2 (n+u)^2}{\varepsilon k(k+1)}} .
\end{equation*}
Our integral then transforms to
\begin{align}
H_k(q) & = \sqrt{\frac{2\pi}{\varepsilon k(k+1)}} \int_{[0, 1]+ic} \sum_{n\in\Z}
\exp\left(-\frac{2\pi^2 (n+u)^2}{\varepsilon k(k+1)} +
\Li_2\left(e^{2 \pi i u};q^k\right)-\Li_2\left(e^{2 \pi i u} q^{k+1}; q^{k+1}\right)\right) du \notag \\
\label{E:Hkqint}
& =\sqrt{\frac{2\pi^2}{\varepsilon k(k+1)}}
\int_{\R+ic} \exp\left( -\frac{2\pi u^2}{\varepsilon k(k+1)}
+\Li_2\left(e^{2\pi iu}; e^{-k\varepsilon}\right)
-\Li_2\left(e^{2\pi iu-(k+1)\varepsilon}; e^{-(k+1)\varepsilon}\right)\right) du.
\end{align}

We now apply the {\it stationary phase} method to this integral, which is
generally useful for determining the asymptotic behavior of integrals of the
form $\int_\R g(u) e^{\frac{i f(u)}{\varepsilon}} du$ as $\varepsilon \to 0^+$.  If
$f$ has a critical point at $u_0$, and both $f$ and $g$ have Taylor
expansions around this point, then the dominant asymptotic term of the
integral can be described in terms of $f(u_0)$ and $g(u_0)$.

In order to apply the stationary phase method to \eqref{E:Hkqint}, we must
first rewrite the integrand in this shape in order to identify the dominant
asymptotic terms. If $|x|<1$ and $B\geq 0,$ then the Laurent expansion of the
quantum dilogarithm begins
\begin{align}
\label{E:Li2expansion}
\Li_2\left(e^{-B\varepsilon} x; e^{-\varepsilon}\right)
&=\sum_{n\geq 1}\frac{x^n e^{-Bn\varepsilon}}{n\left(1-e^{-\varepsilon n}\right)}
=\frac{1}{\varepsilon}\sum_{n\geq 1}\frac{x^n}{n^2}\left(1-n\varepsilon\left(B-\tfrac12\right)+
O\left(\varepsilon^2\right)\right)\\
&=\frac{1}{\varepsilon} \Li_2(x)+\left(B-\tfrac12\right)\log(1-x)+O(\varepsilon), \notag
\end{align}
uniformly in $x$ as $\varepsilon \to 0$.  Here $\Li_2(x) := \sum_{n \geq 1}
\frac{x^n}{n^2}$ is the standard dilogarithm function. Using \eqref{E:Li2expansion},
we study the argument of the exponential in the integrand of
\eqref{E:Hkqint}, namely
\begin{equation}
\label{E:exp}
-\frac{2\pi u^2}{\varepsilon k(k+1)}
+\Li_2\left(e^{2\pi iu}; e^{-k\varepsilon}\right)
-\Li_2\left(e^{2\pi iu-(k+1)\varepsilon}; e^{-(k+1)\varepsilon}\right).
\end{equation}
In particular, we consider the Laurent expansion of \eqref{E:exp} and denote the coefficient of the (leading) $\varepsilon^{-1}$ term by
\begin{equation*}
f(u):=-\frac{2\pi^2 u^2}{k(k+1)}-\frac{\Li_2\big(e^{2\pi iu}\big)}{(k+1)}+\frac{\Li_2\big(e^{2\pi iu}\big)}{k}
= -\frac{2\pi^2 u^2}{k(k+1)} + \frac{\Li_2\big(e^{2\pi iu}\big)}{k(k+1)}.
\end{equation*}

We proceed by determining the critical point(s) of $f$.  Its derivative is
\[
f'(u)
=-\frac{4\pi^2u}{k(k+1)}+\frac{2\pi i\log\big(1-e^{2\pi iu}\big)}{k(k+1)}
=\frac{2\pi i}{k(k+1)}\Big(2\pi iu-\log\Big(1-e^{2\pi iu}\Big)\Big),
\]
and the critical point therefore occurs at
\[
e^{2\pi iu}=1-e^{2\pi iu}, \qquad\text{i.e.,}\qquad u=\frac{i\log
2}{2\pi}.
\]
We denote this critical value of $u$ by $w :=\frac{i\log 2}{2\pi}$, which also determines the most appropriate height of the contour (specifically, $c = \frac{\log{2}}{2\pi}$). At the
critical value, the function $f$ evaluates to
\begin{equation*}
f(w)=\frac{(\log 2)^2}{2k(k+1)}+\frac{1}{k(k+1)}\Li_2\left(\tfrac12\right)=\frac{\pi^2}{12k(k+1)}.
\end{equation*}
Here we used the functional equation \cite{Zag07}
\begin{equation*}
\Li_2(x)+\Li_2(1-x)=\frac{\pi^2}{6}-\log(x)\log(1-x).
\end{equation*}
By definition $f'(w) = 0$, but the second derivative  still makes a
contribution to the overall asymptotic behavior, so we calculate
\begin{equation*}
f''(w)=\frac{2\pi i}{k(k+1)}\left(2\pi i+\frac{2\pi ie^{2\pi iu}}{1-e^{2\pi iu}}\right)=-\frac{8\pi^2}{k(k+1)}<0.
\end{equation*}
We can therefore compute the first terms in the Taylor expansion around $w$
of \eqref{E:exp}.  In particular, using the change of variable
$u=w+\sqrt{\varepsilon}z$, this Taylor expansion is
\begin{equation}
\label{E:fw}
\frac{f(w)}{\varepsilon} + \left(\frac{f''(w)}{2}z^2-
\log\left(1-e^{2\pi iw}\right)\right) + O\left(\varepsilon^{\frac12}\right).
\end{equation}

Plugging in the expansions \eqref{E:Li2expansion} and \eqref{E:fw} to the integral \eqref{E:Hkqint}, the contour may then be shifted to the real axis, leading finally to the evaluation
\begin{align*}
H_k(q)
&=\sqrt{\frac{2\pi}{\varepsilon k(k+1)}}
\,\frac{e^{\frac{f(w)}{\varepsilon}}}{1-e^{2\pi iw}}
\left(1+O\left(\varepsilon^{\frac12}\right)\right)
\int_\R e^{\frac{f''(w)x^2}{2\varepsilon}}dx\\
&=\sqrt{\frac{2\pi}{\varepsilon k(k+1)}}\,
e^{\frac{\pi^2}{12k(k+1)\varepsilon}}\cdot 2\cdot
\sqrt{\frac{2\varepsilon}{-f''(w)}}\sqrt{\pi}\,\left(1+O\left(\varepsilon^{\frac12}\right)\right)\\
&=\sqrt{2}e^{\frac{\pi^2}{12k(k+1)\varepsilon}}\,\left(1+O\left(\varepsilon^{\frac12}\right)\right).
\qedhere
\end{align*}
\end{proof}

\section{Asymptotic behavior of $\pbar{k}(n)$}
\label{S:OverpartnAsymp}
\noindent
In this section we use the asymptotic behavior of generating series in order to determine the asymptotic behavior of $\pbar{k}(n)$ as $n \to \infty$.  As before, we write $q=e^{-\varepsilon}$. Recall that
\begin{equation}\label{recall}
\Gbar{k}(q)=\sum_{n=0}^\infty \pbar{k}(n)q^n=
(q;q)_{\infty}^{-1} \cdot H_k(q).
\end{equation}
Ingham's Tauberian theorem relates the asymptotic behavior of such a series to its coefficients. The following result is a special case of Theorem 1 in \cite{Ing41}.
\begin{theorem}[Ingham]\label{T:Ingham}
Let $f(z)=\sum_{n \geq 0}a(n) z^n$ be a power series with real nonnegative
coefficients and radius of convergence equal to $1$. If there exist $A>0$,
$\lambda, \alpha\in\R$ such that
\[
f(z)\sim
\lambda\left(-\log z\right)^\alpha
\exp\left(\frac{A}{-\log z}\right)
\]
as $z\to 1^-$, then
\[
\sum_{m=0}^n a(m)\sim\frac{\lambda}{2\sqrt{\pi}}\,
\frac{A^{\frac{\alpha}{2}-\frac14}}{n^{\frac{\alpha}{2}+\frac14}}\,
\exp\left(2\sqrt{An}\right)
\]
as $n\to\infty$.
\end{theorem}

\begin{proof}[Proof of Theorem \ref{T:pk}]
The modular inversion formula for Dedekind's eta-function (page 121,
Proposition 14 of \cite{Ko}) states that
\begin{equation*}
(q;q)_{\infty} = \sqrt{\frac{2 \pi}{\varepsilon}} \: e^{-\frac{\varepsilon}{24} - \frac{\pi^2}{6 \varepsilon}}
\prod_{n \geq 1} \Bigl(1 - e^{-\frac{4 \pi^2 n}{\varepsilon}}\Bigr).
\end{equation*}
This implies that $\varepsilon \rightarrow 0^+,$
\begin{equation}\label{E:qinfty}
(q;q)_{\infty}  \sim \sqrt{\frac{2 \pi}{\varepsilon}} \: e^{- \frac{\pi^2}{6\varepsilon}}.
\end{equation}
Combined with Proposition \ref{P:Hkq}, this implies that as $ \varepsilon\to 0^+$
\begin{equation*}
(1-q) \Gbar{k}(q) \sim \frac{\sqrt{\varepsilon}}{\sqrt{\pi}}
\,\exp\left(\frac{\pi^2}{6 \varepsilon}\left(1 + \frac{1}{2k(k+1)}\right)\right).
\end{equation*}
Note that the coefficients of this $q$-series are $(1-q) \Gbar{k}(q) = \sum_{n \geq 0} \big(\pbar{k}(n) - \pbar{k}(n-1)\big)q^n$.  Applying Theorem \ref{T:Ingham} with $a(n):= \pbar{k}(n)-\pbar{k}(n-1)$ then gives the stated asymptotic formula for $\pbar{k}(n)$
by telescoping.
\end{proof}

\end{document}